\pdfoutput=1
\documentclass[sn-mathphys-num]{sn-jnl}

\usepackage{graphicx}%
\usepackage{mathtools}
\usepackage{multirow}%
\usepackage{amsmath,amssymb,amsfonts}%
\usepackage{amsthm}%
\usepackage{mathrsfs}%
\usepackage[title]{appendix}%
\usepackage{xcolor}%
\usepackage{textcomp}%
\usepackage{manyfoot}%
\usepackage{booktabs}%
\usepackage{algorithm}%
\usepackage{algorithmicx}%
\usepackage{algpseudocode}%
\usepackage{listings}%

\def\HH{\mathcal H}

\def\G0{\Gamma_0(\mathcal G)}
\def\H0{\Gamma_0(\mathcal H)}
\def\vv{\boldsymbol{p}}
\definecolor{dred}{HTML}{D90404}
\definecolor{orng}{HTML}{D35400}

\newcommand\blfootnote[1]{%
  \begingroup
  \renewcommand\thefootnote{}\footnote{#1}%
  \addtocounter{footnote}{-1}%
  \endgroup
}
\newcommand{\NN}{\ensuremath{\mathbb N}}
\newcommand{\RR}{\ensuremath{\mathbb{R}}}

\newcommand{\menge}[2]{\big\{{#1}~\big |~{#2}\big\}}

\usepackage{lipsum}
\setcitestyle{square,citesep={,}}

\usepackage{lineno,hyperref}

\renewcommand\theenumi{(\roman{enumi})}

\newcommand{\Argmin}{\ensuremath{{\text{\rm Argmin}}}}

\newcommand{\scal}[2]{{\langle{{#1}\mid{#2}}\rangle}}

\DeclareMathOperator{\proj}{Proj}
\DeclareMathOperator{\lmo}{LMO}

\theoremstyle{thmstyleone}%
\newtheorem{theorem}{Theorem}
\newtheorem{proposition}[theorem]{Proposition}%
\newtheorem{remark}{Remark}%
\newtheorem*{question*}{Question}
\newtheorem*{conjecture*}{Conjecture}
\newtheorem{corollary}{Corollary}
\newtheorem{notation}{Notation}%
\newtheorem{assumption}{Assumption}%
\theoremstyle{thmstyletwo}%
\theoremstyle{thmstylethree}%
\begin{document}

\title[High-precision linear minimization is no slower than projection]{High-precision linear minimization is no
slower than projection}

\author[1]{\fnm{Zev} \sur{Woodstock}}\email{woodstzc@jmu.edu}
\affil[1]{\orgdiv{Department of Mathematics and Statistics},
\orgname{James Madison University}, \orgaddress{\street{MSC 1911,
60 Bluestone Dr.}, \city{Harrisonburg},
\state{VA}, \postcode{22807}, \country{USA}}}
\abstract{
This note demonstrates that, for all compact convex sets,
high-precision linear minimization can be performed via
a single evaluation of the projection and a scalar-vector
multiplication. In consequence, if $\varepsilon$-approximate linear
minimization takes at least $L(\varepsilon)$ real vector-arithmetic
operations
and projection requires $P$ operations, then 
$\mathcal{O}(P)\geq \mathcal{O}(L(\varepsilon))$ is guaranteed. 
This concept is expounded with examples, an explicit error
bound, and an exact linear minimization result for polyhedral sets.}
\maketitle
\section{Introduction}
\blfootnote{The work of Z. Woodstock was supported by the National
Science
Foundation under grant DMS-2532423.}
\begin{notation}
$\HH$ is a real Hilbert space with inner product
$\scal{\cdot}{\cdot}$ and induced norm $\|\cdot\|$. The
set
$C\subset\HH$ is nonempty, convex, and compact. The \emph{normal
cone} and
\emph{indicator function} of $C$ are denoted $N_C$ and $\iota_C$
respectively (see, e.g., \cite{Livre1}). The \emph{projection}
operator
onto $C$ is denoted
$\proj_C\colon\HH\to\HH\colon x\mapsto\Argmin_{c\in C}\|c-x\|$.
The set of \emph{linear minimization oracle} points is
$\lmo_C\colon\HH\to 2^\HH\colon x\mapsto\Argmin_{c\in
C}\scal{c}{x}$. For $\varepsilon\geq 0$, an
\emph{$\varepsilon$-approximate LMO of $x$} is a point $v\in C$ such
that $0\leq \scal{v}{x}-\min_{c\in C}\scal{c}{x}\leq \varepsilon$.
\end{notation}

\begin{assumption}
\label{as:1}
In-line with \cite{Nemi83}, this work assumes the ability to
perform real number arithmetic operations.
Suppose that, for all $x\in\HH$, projection and
$\varepsilon$-approximate linear minimization can be performed over
$C$ using finitely many real vector-arithmetic operations. Let $P$
and $L(\varepsilon)$ respectively denote the smallest amount of
operations required, i.e., the worst-case complexity for an
``optimal'' projection or LMO algorithm.
\footnote{$P$ and $L(\varepsilon)$ exist as limit points of 
monotonic sequences in $\NN$ bounded below.}
\end{assumption}

In constrained first-order optimization, a guiding motivator for the
development of
Frank-Wolfe algorithms is the fact that their hallmark subroutine,
the \emph{linear minimization oracle} (i.e., a selection of the
operator $\lmo_C$), is currently faster than projection
oracles on some sets arising in applications, particularly in
high-dimensional settings \cite{Dunn78,Comb21,Garb21} (also, if
projection is intractable, LMO is used to approximate it \cite{Grun18}). 
While several works suggest that for specific sets,
$P>L(0)$, this principle does not appear to have been definitively
established for all compact convex sets. In fact, regardless of
approximateness ($\varepsilon>0$) or exactness ($\varepsilon=0$),
it appears there are no results pertaining to all compact convex
sets that allow one to compare the computational complexity of
linear minimization to that of projection. 

The main contribution of this article is showing that, for any
$\varepsilon>0$, a high-precision $\varepsilon$-approximate LMO
can be obtained via the use of one projection and a (real)
scalar-vector multiplication, yielding the complexity bound
$P+1\geq L$ (hence, in finite-dimensional settings,
$\mathcal{O}(P)\geq \mathcal{O}(L(\varepsilon))$\footnote{For the
sake of presentation, (A)
we suppress the dependence of $P$ and $L$ on $C$ and the dimension
of $\HH$ and (B) we also only use big-$\mathcal{O}$ notation in
finite-dimensional settings; see \cite{Corm22}.}).
The error
bound in Theorem~\ref{t:1}
explicitly depends on the radius and boundedness of $C$.
It is further demonstrated in finite-dimensional settings that,
when $C$ is also polyhedral, projection is no faster than exact
linear minimization, i.e., the stronger inequality
$\mathcal{O}(P)\geq \mathcal{O}(L(0))$ holds.
The central approximation considered in this article comes from a
geometric concept (see Figure~\ref{fig:1}) that is known
(e.g., see \cite{Mort23}).
Nonetheless, as far as the author is aware, the present error bound
and complexity results appear to be new. 

\section{Relating linear minimization and projection}
We begin with some basic facts; see, e.g., \cite{Livre1} for
further background. Let $x$ and $z$ be points in a
real Hilbert space $\HH$. Then,
\begin{align}
\nonumber
v\in\lmo_C (z)=\underset{c\in C}{\Argmin}\scal{c}{z}
 &
 \Leftrightarrow 0\in
\partial \left(\scal{\cdot}{z}+\iota_C\right)(v)= z + N_C(v)
\Leftrightarrow
-z\in N_C(v)\\
\label{e:lmochar1}
&\Leftrightarrow
\begin{cases}
v\in C\\
\underset{c\in C}{\sup}\scal{-z}{c-v}\leq 0.
\end{cases}
\end{align}
Similarly, we have the following familiar identity
\begin{equation}
\label{e:projchar1}
p=\proj_C(x)\Leftrightarrow 
x-p\in N_C(p) \Leftrightarrow
\begin{cases}
p\in C\\
\underset{c\in C}{\sup} \scal{x-p}{c-p}\leq 0.
\end{cases}
\end{equation}

\begin{proposition}
\label{p:projlmo}
Let $C\subset\HH$ be a nonempty compact convex set. Then, for every
$x\in\HH$,
\begin{equation}
\label{e:1}
\proj_C(x)\in\lmo_C(\proj_C(x)-x).
\end{equation}
\end{proposition}
\begin{proof}
By setting $z=\proj_C(x) - x$ in \eqref{e:lmochar1}, we see from
\eqref{e:projchar1} that $\proj_C(x)$ is a solution of the
characterization in \eqref{e:lmochar1}.
\end{proof}

\begin{remark}
Since linear minimization is not unique in general, evaluating
$\lmo_C(\proj_C(x)-x)$ may not yield $\proj_C(x)$ for every
numerical implementation of $\lmo_C$ \cite{FWjl}. However, for a
particular selection (in the sense of \cite{Livre1}) of the
operator $\lmo_C$, \eqref{e:1} is
guaranteed to hold with equality instead of inclusion.
\end{remark}

For $\lambda>0$ sufficiently large, one can approximate an element
of $\lmo_C(x)$ with
\begin{equation}
\label{e:approx}
\proj_C(-\lambda x).
\end{equation}
\begin{figure}[th]
\centering
\includegraphics[scale=0.18]{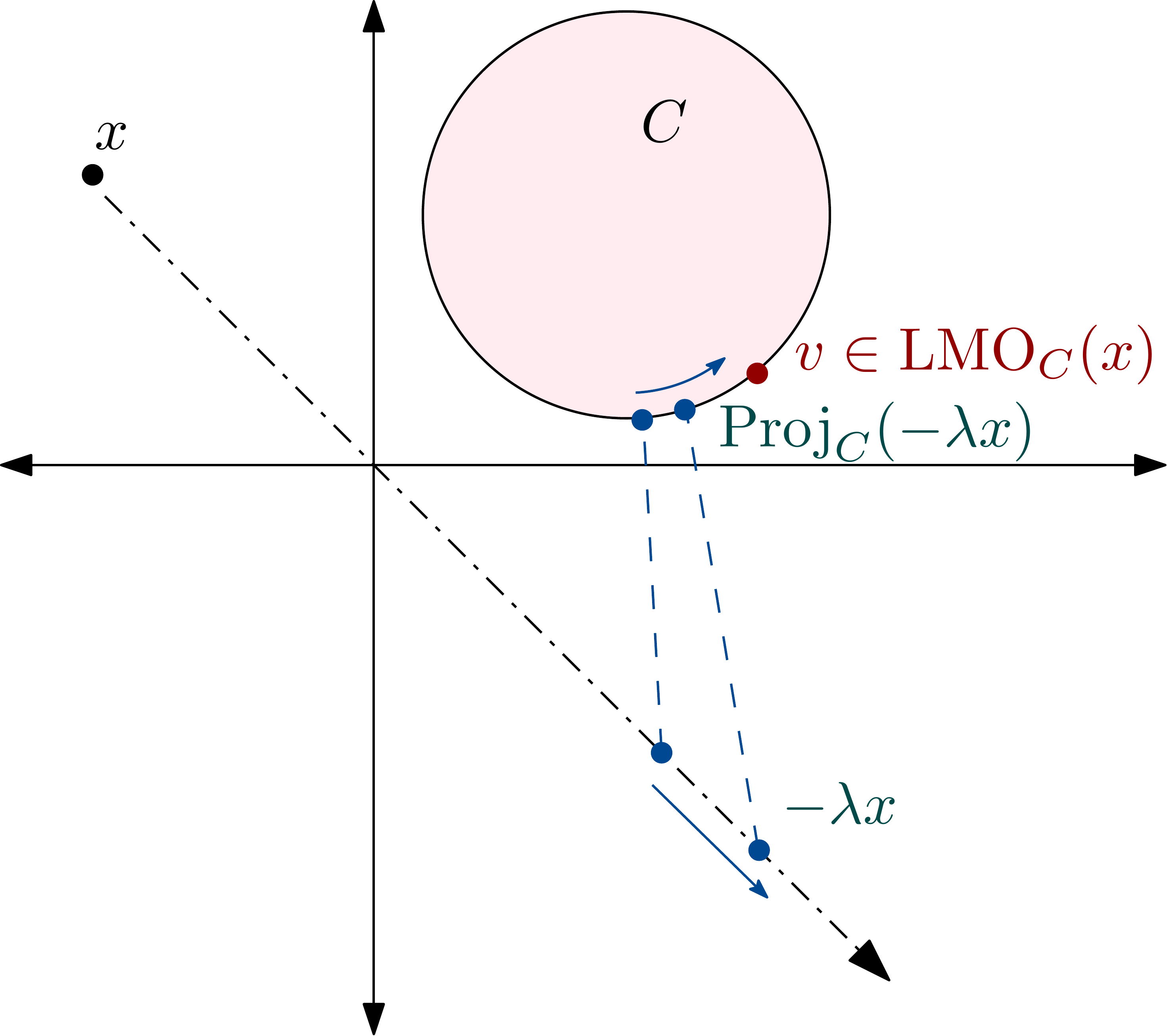}
\caption{As $\lambda$ gets larger, $\proj_C(-\lambda x)$ approaches
$\lmo_C(x)$ for a shifted $\ell_2$ ball.}
\label{fig:1}
\end{figure}

Selecting the parameter $\lambda$ can be guided as follows.
\begin{theorem}
\label{t:1}
Let $x\in\HH$ and let $C$ be a nonempty, compact, and convex subset
of $\HH$ with diameter $\delta_C\coloneqq\sup_{(c_1,c_2)\in
C^2}\|c_1-c_2\|\geq 0$ and bound $\mu_C\coloneqq\sup_{c\in
C}\|c\|\geq 0$. Then, for every $\lambda>0$ and
every $v\in\lmo_C(x)$, 
\begin{equation}
\label{e:422}
0\leq \scal{\proj_C(-\lambda x)}{x} 
-
\min_{c\in C}\scal{c}{x}
\leq
\frac{\|\proj_C(-\lambda x)\|}{\lambda}
\Big(\|v\|-
\|\proj_C(-\lambda x)\|
\Big).
\end{equation}
In consequence,
we have
$\|\proj_C(-\lambda x)\|\leq\|v\|$ and
 for every $\varepsilon>0$, 
\begin{equation}
\label{e:6}
\lambda\geq\frac{\min\left\{\delta_C\mu_C,\mu_C^2\right\}}{\varepsilon}\quad\Rightarrow\quad
0\leq \scal{\proj_C(-\lambda x)}{x} 
-
\underset{c\in C}{\min}\scal{c}{x}
\leq\varepsilon.
\end{equation}
\end{theorem}
\begin{proof}
Let $v\in\lmo_C(x)$.
By Proposition~\ref{p:projlmo}, and the definition of the LMO,
\begin{equation}
\proj_C(-\lambda x)\in \underset{c\in C}{\Argmin}\;\scal{c}{\proj_C(-\lambda x)+\lambda x},
\end{equation}
so, for all $c\in C$,
$\scal{\proj_C(-\lambda x)}{\proj_C(-\lambda x)+\lambda x}\leq
\scal{c}{\proj_C(-\lambda x)+\lambda x}$.
In particular,
\begin{equation}
\scal{\proj_C(-\lambda x)}{\proj_C(-\lambda x)+\lambda x}
\leq
\scal{v }{\proj_C(-\lambda
x)+\lambda x}.
\end{equation}
Dividing by $\lambda$ and rearranging, then proceeding with 
standard norm inequalities yields
\begin{align}
\label{e:399}
\scal{\proj_C(-\lambda x)  }{x}
-
\scal{v }{x}
&
\leq
\lambda^{-1}\big(\scal{v}{\proj_C(-\lambda x)}-\|\proj_C(-\lambda
x)\|^2\big)\\
&
\leq
\lambda^{-1}\big(\|v\|\|\proj_C(-\lambda x)\|-\|\proj_C(-\lambda
x)\|^2\big)\\
&
\label{e:421}
=
\lambda^{-1}\|\proj_C(-\lambda x)\|\big(\|v\|-\|\proj_C(-\lambda
x)\|\big)\\
&\leq
\label{e:433}
\lambda^{-1}\|\proj_C(-\lambda x)\|\|v-\proj_C(-\lambda
x)\|.
\end{align}
Since $v$ is optimal, we can also lower bound \eqref{e:399} by $0$;
hence \eqref{e:422} follows from \eqref{e:421}.
An immediate consequence of \eqref{e:422} is that
$\|\proj_C(-\lambda x)\|\leq\|v\|$. Proceeding with
the bounds in \eqref{e:433}, we have
\begin{equation}
\label{e:423}
0\leq 
\scal{\proj_C(-\lambda x)}{x}
-
\scal{v}{x}
\leq
\lambda^{-1}\delta_C\mu_C.
\end{equation}
On the other hand, dropping the negative term in \eqref{e:421}
implies $0\leq \scal{\proj_C(-\lambda x)}{x} - \scal{v}{x}, \leq
\lambda^{-1}\mu_C^2$. Hence, in either case of
$\lambda\geq\varepsilon^{-1}\min\{\delta_C\mu_C,\mu_C^2\}$, \eqref{e:6} holds. In
particular,
$\proj_C(-\lambda x)$ is an $\varepsilon$-approximate LMO.
\end{proof}

\begin{corollary}[Projection is no faster than approximate LMO]
Let $\varepsilon>0$ and suppose that Assumption~\ref{as:1} holds.
Then $P+1\geq L(\varepsilon)$. In consequence, if $P\geq 1$ and
$\HH$ is finite-dimensional, then $\mathcal{O}(P)\geq
\mathcal{O}(L(\varepsilon))$.
\end{corollary}
\begin{proof}
From Theorem~\ref{t:1}, $L(\varepsilon)$ is bounded above by the cost of
evaluating $\proj_C(-\lambda x)$, which is $P+1$.
\end{proof}

\begin{remark}
The type of bound provided by Theorem~\ref{t:1} is consistent with
many algorithms that allow for inexact
linear minimization, e.g., \cite{Dunn78,Jagg13,Freu16,Silv21}.
\end{remark}
In general, Theorem~\ref{t:1} requires $\lambda\to\infty$ to drive
the error bound to zero (as demonstrated, e.g., for a shifted
$\ell_2$ ball). However, for some sets (e.g., the 
$\ell_{\infty}$ ball), exact minimization is achieved for finite
$\lambda$, i.e.,
$\proj_C(-\lambda x)\in\lmo_C(x)$.\footnote{
Interactive graph demonstrations in
2-D for the shifted
$\ell_2$ ball (see also Fig.~\ref{fig:1}):
\url{https://www.desmos.com/calculator/ntpe1pncpu} and
 $\ell_{\infty}$ ball:
\url{https://www.desmos.com/calculator/qk3tnqskgw}.
}
As will be seen in Proposition~\ref{p:p},
exact linear minimization for finite $\lambda$
can be achieved more generally.
\begin{proposition}[Projection is no faster than exact LMO on
polyhedral sets]
\label{p:p}
Let $x\in\RR^n\eqqcolon\HH$ and suppose that
$C\subset\HH$
is compact, convex,
and polyhedral. Then
there exists a finite value $\lambda^*\geq 0$ such that
$\proj_C(-\lambda^* x)\in\lmo_C(x)$. Further, if
Assumption~\ref{as:1} holds, then $P+1\geq L(0)$; if $P\geq 1$,
then $\mathcal{O}(P)\geq \mathcal{O}(L(0))$.
\end{proposition}
\begin{proof}
Without loss of generality $C=\menge{x\in\HH}{Ax\leq b}$ where
$A\in\mathbb{R}^{m\times n}$ and $b\in\mathbb{R}^m$.
Let $\nu=\min_{c\in C}\scal{c}{x}$ and consider the problem of
computing $\proj_{\lmo_C(x)}(\boldsymbol{0})$, i.e., the minimal-norm
element of $\lmo_C(x)$:
\begin{equation}
\label{e:plmo}
\underset{\substack{z\in\HH\\Az\leq b\\ \scal{z}{x}\leq
\nu}}{\textrm{minimize}}\;\;
{\frac{1}{2}\|z\|^2}.
\end{equation}
By \cite[Theorem~11.15]{Gule10}, strong duality holds and hence the
perturbation function
$\vv\colon\mathbb{R}^{m+1}\to\left[-\infty,+\infty\right]\colon
y\mapsto\inf_{z\in\HH}\{f(x)\,|\,Az-b\leq 
(y_i)_{i=1}^m;\;\;\scal{z}{x}-\nu\leq
y_{m+1}\}$ is \emph{stable} in the sense of \cite{Geof71}, i.e.,
$\vv(\boldsymbol{0})$ is finite and there exists $M>0$ such that
\cite[pp.~8]{Geof71}
\begin{equation}
\label{e:019}
(\forall \xi>0)\quad \frac{\vv(\boldsymbol{0})-\vv(\xi,\ldots,\xi)}{\xi}\leq M.
\end{equation}
With an eye towards using \cite[Theorem~3]{Geof71}, we will use
\eqref{e:019} to show that
the partially-dualized perturbation function
$p\colon\mathbb{R}\to\left[-\infty,+\infty\right]\colon
\xi\mapsto\inf_{z\in\HH; Az-b\leq 0}\{f(x)\,|\,\scal{x}{z}-\nu\leq
\xi\}$ is also stable: this follows from the fact that
$p(0)=\vv(\boldsymbol{0})$ and, for all $\xi>0$, $p(\xi)\geq
\vv(\xi,\ldots,\xi)$, so
\begin{equation}
\frac{p(0)-p(\xi)}{\xi}=\frac{\vv(\boldsymbol{0})-p(\xi)}{\xi}\leq
\frac{\vv(\boldsymbol{0})-\vv(\xi,\ldots,\xi))}{\xi}\leq M,
\end{equation}
as claimed.
By \cite[Theorem~3]{Geof71}, there exists $\lambda^*\geq 0$
such that
\begin{align}
\proj_{\lmo_C(x)}(\boldsymbol{0})
&=
\underset{\substack{z\in\HH\\Az\leq b}}{\Argmin}\,
\frac{1}{2}\|z\|^2+\lambda^*(\scal{x}{z}-\nu)\\
\label{e:19}
&=\underset{\substack{z\in\HH\\Az\leq b}}{\Argmin}\,
\frac{1}{2}\|-\lambda^* x-z\|^2=\proj_C(-\lambda^* x),
\end{align}
where \eqref{e:19} makes use of the fact that the minimization is
not changed by addition of the constant $\lambda^*\nu+\|\lambda^*x\|^2/2$.
This establishes that, for finite $\lambda^*\geq 0$,
$\proj_C(-\lambda^* x)\in\lmo_C(x)$. In consequence, under
Assumption~\ref{as:1}, the computation requied to
perform $\proj_C(-\lambda^* x)$ (namely $P+1$) is an upper
bound on $L(0)$, which completes the proof.
\end{proof}

\section{Conclusion}

\begin{remark}
In practice, using $\proj_{C}(-\lambda x)$ to approximate an LMO 
point is not advisable since certain
constants ($\delta_C$ and $\mu_C$ for Theorem~\ref{t:1}; $\lambda^*$
for Proposition~\ref{p:p}) may be unknown, and faster methods may
be available
\cite{Dunn78,Comb21,Garb21}.
The main utility of these results are to establish a relationship
between the optimal complexities $P$ and $L(\varepsilon)$ -- not to
provide an efficient algorithm for linear minimization.
\end{remark}

While this note demonstrates two connections between the
complexities of projection and linear minimization, the
general relationship between $P$ and $L(0)$ warrants further
investigation. 

For some sets (e.g., singletons), $P=L(0)$. Hence, even though
current evidence suggests that
$P>L(0)$ on a myriad of important sets
\cite{Comb21}, strict inequality cannot hold in general.
Nonetheless, it remains an open question as to whether or not there exists {\em
any}
compact convex set such that its projection operator has a faster
optimal runtime complexity than exact linear minimization.

\begin{question*}
Does there exist a nonempty compact convex set such that
$P<L(0)$?
\end{question*}

\begin{conjecture*}
If the dimension of $\HH$ is finite, then $L(0)$ is $\mathcal{O}(P)$.
\end{conjecture*}

\subsubsection*{Acknowledgements}
The work of Z. Woodstock was supported by the National
Science
Foundation under grant DMS-2532423.

Many thanks to Mathieu Besan\c con, who first suggested trying a
duality approach for the proof of Proposition~\ref{p:p}. Thanks to
two anonymous referees, Jannis Halbey, Sebastian Pokutta, Elias
Wirth, and the anonymous user `gerw' (on Math Stack Exchange) for
early correspondences on this topic.

\subsubsection*{Data availability statement}
No data was used in the creation of this article. Citations are
available above.

\subsubsection*{Conflict of interest statement}
The author declares no conflict of interest.


\begin{thebibliography}{1}
{\scriptsize
\bibitem{FWjl}
M. Besan\c con, M. Carderera, and S. Pokutta,
FrankWolfe. jl: A high-performance and flexible toolbox for
Frank--Wolfe algorithms and conditional gradients,
{\em INFORMS J. Comput.},
vol.~34 (5),
pp. 2611--2620,
2022.


\bibitem{Livre1}
H. H. Bauschke and P. L. Combettes, 
{\em Convex Analysis and Monotone Operator Theory in Hilbert 
Spaces}, 2nd ed. 
Springer, New York, 2017.

\bibitem{Comb21}
C. W. Combettes and S. Pokutta,
Complexity of linear minimization and projection
on some sets,
{\em Oper. Res. Lett.},
vol. 49 (4),
pp. 565--571,
2021.

\bibitem{Corm22}
T. H. Cormen, C. E. Leiserson, R. L. Rivest, and
C. Stein,
{\em Introduction to Algorithms}, 4th ed.
MIT Press, Cambridge, 2022.

\bibitem{Dunn78}
J. Dunn and S. Harshbarger,
Conditional gradient algorithms with open loop step size rules.
{\em J. Math.  Anal. Appl.}, 
vol. 62, pp. 432--444, 1978.

\bibitem{Freu16}
R. M. Freund and P. Grigas,
New analysis and results for the Frank-Wolfe method,
{\em Math. Program., Ser. A},
vol. 155, pp. 199--230, 2016.


\bibitem{Garb21}
D. Garber, A. Kaplan, and S. Sabach,
Improved complexities of conditional gradient-type methods with
applications to robust matrix recovery problems,
{\em Math. Program.},
vol. 186, pp. 185--208, 2021.

\bibitem{Geof71}
A. M. Geoffrion,
Duality in nonlinear programming: A simplified
applications-oriented development,
{\em SIAM Rev.},
vol. 13 (1) pp. 1--37, 1971.

\bibitem{Grun18}
S. Gr\" unewalder,
Compact Convex Projections,
{\em J. Mach. Learn. Res.},
vol. 18 (219), pp. 1--43, 2018.


\bibitem{Gule10}
O. G\"uler
{\em Foundations of Optimization},
Springer, New York, 2010.

\bibitem{Jagg13}
M. Jaggi,
Revisiting Frank-Wolfe: Projection-free sparse convex optimization.
{\em Proc. 30th International Conference on Machine
Learning}, in PMLR,
vol. 28 (1), pp. 427--435, 2013.

\bibitem{Mort23}
H. Mortagy, S. Gupta, and S. Pokutta,
Walking in the shadow: A new perspective on
descent directions for constrained minimization,
{\em Proc. 34th Conference on Neural Information Processing
Systems}, 
vol. 33, pp. 12873--12883, 2020.


\bibitem{Nemi83}
A. S. Nemirovskij and D. B. Yudin,
{\em Problem complexity and method efficiency in optimization},
John Wiley \& Sons, 1983.

\bibitem{Silv21}
A. Silveti-Falls, C. Molinari and J. Fadili,
Inexact and stochastic generalized conditional gradient with
augmented Lagrangian and proximal step,
{\em J. Nonsmooth Anal. Optim.},
vol.~2, 2021.
}
\end{thebibliography}
\end{document}